\documentclass[a4paper,10pt,reqno, english]{amsart}

\usepackage{amsmath,amssymb,amscd,amsthm,amsfonts}
\usepackage{graphicx,subfigure}
\usepackage{hyperref}
\usepackage{dsfont}
\usepackage[nobysame, alphabetic]{amsrefs}
\usepackage{tikz}
\usepackage[capitalise]{cleveref}

\newtheorem{theorem}{Theorem}
\newtheorem{lemma}{Lemma}

\newtheorem{conjecture}{Conjecture}

\newtheorem{definition}{Definition}

\def\rr{\mathds{R}}

\DeclareMathOperator{\conv}{conv}

\title{An elementary proof of Sierksma's conjecture for seven points in the plane}



\author[Sober\'on]{Pablo Sober\'on}\address{Baruch College and The Graduate Center, City University of New York, One Bernard Baruch Way, New York, NY 10010, United States} 
\email{psoberon@gc.cuny.edu}

\thanks{
The research of P. Sober\'on is supported by NSF CAREER grant DMS-237324 and a PSC-CUNY Trad B award.}

\begin{document}

\maketitle

\begin{abstract}
We give a new simple geometric proof that any seven points in the plane have four Tverberg partitions into three sets.  This is the only confirmed non-trivial case of Sierksma's conjecture.  Earlier proofs, by Stephan Hell, relied on topological arguments.
\end{abstract}

%\tableofcontents

\section{Introduction}

In this note, we revisit a structural claim about Tverberg's theorem in combinatorial geometry.  Tverberg's theorem is a central result in the area that describes how one can partition sets of points so that their convex hulls intersect.  There is a significant research interest in understanding Tverberg's theorem and its generalizations \cites{Barany2018, DeLoera2019, Barany2022}.

\begin{theorem}[Tverberg 1966, \cite{Tverberg:1966tb}]
    Let $r,d$ be positive integers.  Given $(r-1)(d+1)+1$ points in $\rr^d$ there exists a partition of them into $r$ parts whose convex hulls intersect.
\end{theorem}

We call a partition as above a Tverberg partition.  A notorious open problem around Tverberg's theorem is a conjecture by Sierksma concerning the number of Tverberg partitions of a set of $(r-1)(d+1)+1$ points.

\begin{conjecture}[Sierksma 1979 \cite{Sierksma1979}]
    Every set of $(r-1)(d+1)+1$ points in $\rr^d$ has at least $(r-1)!^d$ Tverberg partitions into $r$ parts.
\end{conjecture}

The conjecture above has remained open for almost fifty years.  Currently, the best general lower bound for the number of Tverberg partitions is roughly $(r-1)!^{d/2}$ when $r$ is a prime number by Vu\v{c}i\'c and \v{Z}ivaljevi\'c \cite{Vucic1993} and later extended to prime powers by Hell \cite{Hell2007}.  The only non-trivial case of the conjecture that has been confirmed is $r=3, d=2$ by Hell \cite{Hell2008} (see also Hell's PhD thesis, which suggests this proof \cite{Hell2006}*{Rmk 3.19}).  Hell's proof uses advanced topological machinery.  There are several constructions in which the number of Tverberg partitions matches exactly Sierksma's conjecture \cites{White2017, Bukh2017, por2018universality}.

In this note, we present a new proof of the case $r=3,d=2$ of Sierksma's conjecture using only elementary geometric arguments.

\begin{theorem}[Hell 2008]\label{thm:main}
    Let $X$ be a set of seven points in the plane.  There are at least four Tverberg partitions of $X$ into three parts.
\end{theorem}

We hope that an additional proof of this case will motivate further work around Sierksma's conjecture.  Our proof relies on arguments similar to Birch's original proof of Tverberg's theorem \cite{Birch1959}.

In \cref{sec:prelim} we establish some preliminary lemmas before proving \cref{thm:main} in \cref{sec:proof}.  Finally, we discuss which parts of the proof can easily be generalized for higher values of $r$ and $d$ in \cref{sec:remarks}.

\section{Preliminary lemmas}\label{sec:prelim}

For seven points in general position in the plane, there are only two kinds of Tverberg partitions.  The first is two triples and a common point of their convex hull, and the second is two intersecting segments and a triangle that contains the point of intersection.  We call the first kind of partition a $(3,3,1)$ partition, and the second a $(3,2,2)$ partition.

\begin{definition}
    For a set $X \subset \rr^2$ we define $C_k(X)$ to be the set of points of Tukey depth at least $k$.  Formally,
    \[
    C_k (X) = \{ p \in \rr^2: \mbox{ every closed half-plane $H$ with $p \in H$ satisfies $|X \cap H|\ge k$}\}.
    \]
\end{definition}

\begin{lemma}
    Let $X$ be a set of seven points in the plane in general position.  Then $C_3(X)$ is either a single point or $2$-dimensional.  Moreover, if $C_3(X)$ is a point, it must be a point of $X$.
\end{lemma}

\begin{proof}
    Let $C=C_3(X)$.  The set $C$ can be written as the intersection of a finite number of closed half-planes.  Moreover, a minimal family $\mathcal{F}$ of half-planes that determine $C$ is one where their bounding lines each pass through exactly two points of $X$ and contain exactly one or two more points of $X$ in one of its complementary open half-plane.  Assume for a contradiction that $C$ is a segment $I$.  Let $x$ be a point in the relative interior of that segment.  The only half-planes that determine $C$ and contain $x$ must be parallel to $I$.  Since $C$ is one-dimensional, the two complementary closed half-planes that contain $I$ must be in $\mathcal{F}$.  This is a contradiction since their complements would contain at most four points of $X$, and the line containing $I$ has at most two other points of $X$, due to the general position argument, which implies $|X|\le 6$.

    Now assume again for a contradiction that $C$ is a point $p \not\in X$.  Consider $\mathcal{F}'\subset \mathcal{F}$ the set of bounding hyperplanes of $C$ whose bounding line contains $p$.  Since $C=\{p\}$, there must be three such half-planes whose union is $\rr^2$.  The union of their open complements is $\rr^2 \setminus\{p\}$.  However, each contains at most two points of $X$, which would imply again $|X| \le 6$, a contradiction.
\end{proof}

\begin{lemma}\label{lem:vertices-of-C}
    Let $X$ be a set of seven points in the plane in general position.  If $C=C_3(X)$ is $2$-dimensional, then we can assign to every vertex of $C$ that is not in $X$ a different Tverberg $(3,2,2)$ partition of $X$.
\end{lemma}

\begin{proof}
    Suppose that vertex $v$ is the intersection of two sides $\ell_1$ and $\ell_2$ of $C$.  The fact that $\ell_1$ is a side of $C$ implies that the line containing $\ell_1$ must contain two points of $X$.  Additionally, its open half-plane that does not contain $C$ has at most two points of $X$.  Moreover, $v$ must be between the two points of $X$ in this line.  If not, a small rotation of this line around $v$ would show that $v \not\in C$.  Let $\{a,b\}$ be the two points of $X$ in the line determined by $\ell_1$, and $\{c,d\}$ be the two points of $X$ in the line determined by $\ell_2$.  Let $\{e,f,g\}$ be the rest of the points in $X$.  We claim that $v \in \conv \{e,f,g\}$.  If not, then there is a line through $v$ that leaves $\{e,f,g\}$ on one side.  We can assume without loss of generality that this line does not contain $\ell_1$ nor $\ell_2$.  The other side of this line would contain exactly one point of $\{a,b\}$ and exactly one point of $\{c,d\}$, contradicting $v \in C$.  We assign to $v$ the Tverberg partition $\{a,b\}, \{c,d\},\{e,f,g\}$.  Any two vertices of $C$ will receive different partitions.
\end{proof}

\begin{lemma}\label{lemma:3-3-1}
    Let $X$ be a set of seven points in general position such that one of the points $x \in X$ satisfies $x \in C_3(X)$.  Then, there are at least two $(3,3,1)$ Tverberg partitions of $X$ in which $x$ is the singleton. 
\end{lemma}

\begin{proof}
    The condition is equivalent to $x \in C_2(Y)$, where $Y = X \setminus\{x\}$.  Order the points of $Y$ clockwise around $x$ as $p_1,\dots, p_6$.  Birch's classic argument shows that $\{p_1, p_3, p_5\}, \{p_2,p_4,p_6\}, \{x\}$ is a Tverberg partition.  This is because any half-plane that contains $x$ must have at least two consecutive points of $Y$ in the cyclic order, which implies it has at least one point of $p_1,p_3,p_5$ and at least one point of $p_2,p_4,p_6$.  Therefore, we cannot have a line separating $x$ from either of those two sets.  We need to show that there is another $(3,3,1)$ partition.  This follows from Hell's results on Birch partitions \cite{Hell2008a}, but in this case we can give a simple proof as well.

    Whether a triangle $\triangle p_ip_jp_k$ contains $x$ only depends on the directions of the rays $xp_i, xp_j, xp_k$, so we may assume without loss of generality that $p_1,\dots,p_6$ are in convex position.  To show the existence of a second $(3,3,1)$ partition, we use a method similar to the solution of Problem 6 from IMO 2006.  Consider the oriented lines $p_i p_{i+3}$ with indices modulo $6$.  As we increase the value of $i$ by $1$, we eventually go from $p_1p_4$ to $p_4p_1$, which means that the side of the line on which $x$ lies changes.  Therefore, there must be an $i$ such that $x$ is left of $p_ip_{i+3}$ but on the right of $p_{i+1}p_{i+4}$.  Assume without loss of generality that $i=1$ (see \cref{fig:rotate}).  Now, the triangle $\triangle p_6p_1p_2$ does not contain $x$ since $x \in C_2(Y)$, so $x \in \triangle p_2p_5p_6$.  Similarly, $x \not\in \triangle p_1 p_2 p_3$, so $x \in \triangle p_1p_3p_4$.  This gives us the desired second partition $\{x\},\{p_1,p_3,p_4\}, \{p_2,p_5,p_6\}$.

    \begin{figure}
        \centering
        \includegraphics[scale=1]{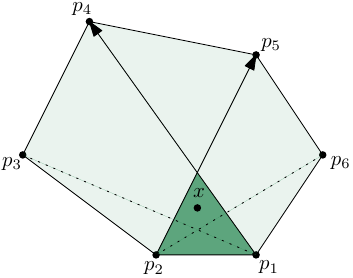}
        \caption{If $x$ is to the left of $p_1p_4$ and to the right of $p_2p_5$ it must be in the shaded triangle.  Then, the two dotted lines must leave $x$ on the side with the most points from $Y$.}
        \label{fig:rotate}
    \end{figure}
\end{proof}

\section{Proof of \cref{thm:main}}\label{sec:proof}

We can assume without loss of generality that $X$ is in general position.  If $X$ is not in general position, it can be approximated by a sequence $\{X_n\}$ of points in general position.  If every $X_n$ has at least four Tverberg partitions into three parts, then the limit $X$ does too.

We look at several cases depending on the structure of $C=C_3(X)$.

\textbf{Case 1.}  $C \cap X = \emptyset$.

In this case, $C$ is $2$-dimensional, which implies it is a convex polygon.  We claim that $C$ has at least four vertices.  If this claim holds, then by \cref{lem:vertices-of-C} we have the four Tverberg partitions we seek.  If $C$ is a triangle, consider the three open half-planes defined by the sides of $C$ that do not contain $C$.  By construction, each has at most two points of $X$.  This implies $|X| \le 6$, a contradiction.

\textbf{Case 2.} $|C \cap X| \ge 2$.

In this case, by \cref{lemma:3-3-1}, each point of $C \cap X$ defines at least two different $(3,3,1)$ Tverberg partitions, and we are done.

\textbf{Case 3. $C$ is $2$-dimensional and $|C \cap X| = 1$.}

In this case, the point in $X \cap C$ defines two $(3,3,1)$ Tverberg partitions.  We have at least two vertices of $C$ that are not in $X$, which gives us two additional $(3,2,2)$ Tverberg partitions.

\textbf{Case 4. $C$ is a single point of $X$.}

Let $C = \{x\}$ with $x \in X$.    Since $C = \{x\}$ there must be three half-planes $H_1, H_2, H_3$ whose union equals $\rr^2$, each of which has $x$ in its boundary, and each of which contains exactly $5$ points of $X$ (one of which is $x$).  The three boundary lines of $H_1, H_2, H_3$ divide $\rr^2$ into six sectors.  Three sectors are covered twice and three are covered once, see \cref{fig:sectors}.

\begin{figure}
    \centering
    \includegraphics[scale=1]{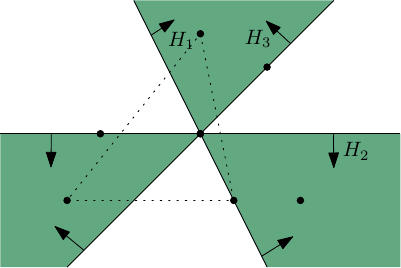}
    \caption{Example of the three doubly-covered sectors.  It is impossible to separate $x$ from a triangle with a vertex on each sector since the side that contains $x$ would also contain a full doubly-covered sector.}
    \label{fig:sectors}
\end{figure}

Let $Y = X \setminus\{x\}$.  We have $|H_1 \cap Y| + |H_2 \cap Y| + |H_3 \cap Y| = 12 = 2 |Y|$.  Therefore, each point of $Y$ must be in the doubly-covered sectors.  Moreover, each of those sectors must have exactly two points of $Y$, as otherwise there would be at least one with fewer than two points, and we would contradict $x \in C_2(Y)$.  Finally, as the doubly-covered sectors are alternating, choosing one point from each sector will give us the vertices of a triangle that contains $x$.  This gives us four $(3,3,1)$ Tverberg partitions with $x$ as the singleton, as it is the number of ways to split $Y$ into two triples that have exactly one point in each doubly-covered sector.

This concludes the proof of \cref{thm:main}.

\section{Remarks}\label{sec:remarks}

It is tempting to extend the ideas to higher values of $r$ and $d$.  Let us mention which parts of the proof can be extended.

First, if we want to count the number of Tverberg partitions of $X$ into $r$ parts, we need to work with $C_r(X)$.  In dimensions greater than or equal to $3$, the set $C_r(X)$ may properly contain the set of Tverberg points for $r$-partitions \cite{Avis1993}, so relying on the vertices of $C_r(X)$ may not be sufficient.  However, it is encouraging that Case $4$ can be extended without problems.

\textbf{Case 4.}  If $X \subset \rr^d$ is a set of $(r-1)(d+1)+1$ points and $C_r(X)$ is a single point $x$, the same arguments show that $x\in X$.  Moreover, there will be a set of $d+1$ closed half-spaces, each containing exactly $d(r-1)+1$ points of $X$ (including $x$) that cover $\rr^d$.  There will be $d+1$ sectors that are covered $d$ times by these half-spaces.  The same counting argument shows that $X\setminus\{x\}$ must be evenly distributed among these $d+1$ sectors.  Picking one point of each sector again gives a simplex that contains $x$, so we have $(r-1)!^d$ partitions.

When $C_r(X)$ is higher-dimensional, only a handful of the properties we used are preserved.  For example, consider the case when $X \subset \rr^d$ is a set of $(r-1)(d+1)+1$, $C_r(X)$ has dimension $d$, and $C_r(X) \cap X = \emptyset$.  Using the same method, we can show that $C_r(X)$ cannot be a simplex, so it must have more than $d+1$ vertices.  Even in the case $d=2$, giving a lower bound for the number of Tverberg partitions associated with each vertex is difficult.  See Hell's results for the best bounds of this kind \cite{Hell2008a}.  Hell showed that for any set $X$ of $3N$ points in the plane, any point $x$ in $C_N(X)$ will be in at least $N!$ different $(N+1)$-Tverberg partitions of $X \cup \{x\}$ in which $x$ is a singleton and $X$ is partitioned into triples.

Combining our methods with Hell's results on the number of Birch partitions only shows that for a set $X$ of $3r-2$ points in the plane, there are at least $4(r-2)!$ Tverberg partitions of $X$ into $r$ parts.  This is still far from the $(r-1)!^2$ conjectured by Sierksma.


\begin{bibdiv}
\begin{biblist}

\bib{Avis1993}{article}{
      author={Avis, David},
       title={The m-core properly contains the m-divisible points in space},
        date={1993},
     journal={Pattern recognition letters},
      volume={14},
      number={9},
       pages={703\ndash 705},
}

\bib{Birch1959}{article}{
      author={Birch, B.~J.},
       title={On {$3N$} points in a plane},
        date={1959},
        ISSN={0008-1981},
     journal={Proc. Cambridge Philos. Soc.},
      volume={55},
       pages={289\ndash 293},
         url={https://doi.org/10.1017/s0305004100034071},
      review={\MR{109315}},
}

\bib{Barany2022}{article}{
      author={B\'ar\'any, Imre},
      author={Kalai, Gil},
       title={Helly-type problems},
        date={2022},
        ISSN={0273-0979,1088-9485},
     journal={Bull. Amer. Math. Soc. (N.S.)},
      volume={59},
      number={4},
       pages={471\ndash 502},
         url={https://doi.org/10.1090/bull/1753},
      review={\MR{4478031}},
}

\bib{Bukh2017}{article}{
      author={Bukh, Boris},
      author={Loh, Po-Shen},
      author={Nivasch, Gabriel},
       title={Classifying unavoidable {T}verberg partitions},
        date={2017},
        ISSN={1920-180X},
     journal={J. Comput. Geom.},
      volume={8},
      number={1},
       pages={174\ndash 205},
         url={https://doi.org/10.20382/jocg.v8i1a9},
      review={\MR{3670821}},
}

\bib{Barany2018}{article}{
      author={B\'ar\'any, Imre},
      author={Sober\'on, Pablo},
       title={Tverberg's theorem is 50 years old: a survey},
        date={2018},
        ISSN={0273-0979,1088-9485},
     journal={Bull. Amer. Math. Soc. (N.S.)},
      volume={55},
      number={4},
       pages={459\ndash 492},
         url={https://doi.org/10.1090/bull/1634},
}

\bib{DeLoera2019}{article}{
      author={De~Loera, Jes\'us~A.},
      author={Goaoc, Xavier},
      author={Meunier, Fr\'ed\'eric},
      author={Mustafa, Nabil~H.},
       title={The discrete yet ubiquitous theorems of {C}arath\'eodory,
  {H}elly, {S}perner, {T}ucker, and {T}verberg},
        date={2019},
        ISSN={0273-0979,1088-9485},
     journal={Bull. Amer. Math. Soc. (N.S.)},
      volume={56},
      number={3},
       pages={415\ndash 511},
         url={https://doi.org/10.1090/bull/1653},
}

\bib{Hell2006}{thesis}{
      author={Hell, Stephan},
       title={Tverberg-type theorems and the fractional helly property},
        type={Dissertation, Technische Universit{\"a}t Berlin},
     address={Berlin, Germany},
        date={2006},
}

\bib{Hell2007}{article}{
      author={Hell, Stephan},
       title={On the number of {T}verberg partitions in the prime power case},
        date={2007},
        ISSN={0195-6698,1095-9971},
     journal={European J. Combin.},
      volume={28},
      number={1},
       pages={347\ndash 355},
         url={https://doi.org/10.1016/j.ejc.2005.06.005},
      review={\MR{2261824}},
}

\bib{Hell2008a}{article}{
      author={Hell, Stephan},
       title={On the number of {B}irch partitions},
        date={2008},
        ISSN={0179-5376,1432-0444},
     journal={Discrete Comput. Geom.},
      volume={40},
      number={4},
       pages={586\ndash 594},
         url={https://doi.org/10.1007/s00454-008-9083-9},
      review={\MR{2453329}},
}

\bib{Hell2008}{article}{
      author={Hell, Stephan},
       title={Tverberg's theorem with constraints},
        date={2008},
        ISSN={0097-3165,1096-0899},
     journal={J. Combin. Theory Ser. A},
      volume={115},
      number={8},
       pages={1402\ndash 1416},
         url={https://doi.org/10.1016/j.jcta.2008.02.007},
      review={\MR{2455585}},
}

\bib{por2018universality}{article}{
      author={P{\'o}r, Attila},
       title={Universality of vector sequences and universality of tverberg
  partitions},
        date={2018},
     journal={arXiv preprint arXiv:1805.07197},
}

\bib{Sierksma1979}{book}{
      author={Sierksma, Gerard},
       title={{Convexity without linearity; the Dutch cheese problem}},
        date={1979},
        note={{Mimeographed notes}},
}

\bib{Tverberg:1966tb}{article}{
      author={Tverberg, Helge},
       title={{A generalization of Radon's theorem}},
        date={1966},
     journal={J. London Math. Soc.},
      volume={41},
      number={1},
       pages={123\ndash 128},
}

\bib{Vucic1993}{article}{
      author={Vu\v{c}i\'{c}, Aleksandar},
      author={{\v{Z}}ivaljevi\'c, Rade~T.},
       title={Note on a conjecture of {S}ierksma},
        date={1993},
        ISSN={0179-5376,1432-0444},
     journal={Discrete Comput. Geom.},
      volume={9},
      number={4},
       pages={339\ndash 349},
         url={https://doi.org/10.1007/BF02189327},
}

\bib{White2017}{article}{
      author={White, Moshe~J.},
       title={On {T}verberg partitions},
        date={2017},
        ISSN={0021-2172,1565-8511},
     journal={Israel J. Math.},
      volume={219},
      number={2},
       pages={549\ndash 553},
         url={https://doi.org/10.1007/s11856-017-1490-2},
      review={\MR{3649599}},
}

\end{biblist}
\end{bibdiv}
\end{document}